\documentclass[12pt]{amsart}
\usepackage[utf8]{inputenc}
\usepackage{amssymb}
\usepackage{amssymb,amsfonts,amsmath}
\usepackage{amsthm}
\usepackage{centernot}
\usepackage{mathtools}
\usepackage{mathrsfs}
\usepackage{dsfont}
\usepackage{verbatim}
\usepackage{cleveref}

\usepackage{enumerate}

\makeatletter
\@namedef{subjclassname@2010}{%
  \textup{2010} Mathematics Subject Classification}
\makeatother

\newtheorem{thm}{Theorem}[section]
\newtheorem{cor}[thm]{Corollary}
\newtheorem{lem}[thm]{Lemma}
\newtheorem{prop}[thm]{Proposition}

\theoremstyle{definition}
\newtheorem{defin}[thm]{Definition}
\newtheorem{rem}[thm]{Remark}

\numberwithin{equation}{section}

\newcommand{\R}{\mathbb R} 
\newcommand{\N}{\mathbb N} 
\newcommand{\topf}{\mathsf{t}} 
\newcommand{\botf}{\mathsf{b}} 
\newcommand{\var}{\mathsf{var}} 
\newcommand{\dH}{\mathsf{d_H}} 
\newcommand{\cht}{\mathds 1} 
\newcommand{\abs}[1]{\left\lvert#1\right\rvert} 
\newcommand{\norm}[1]{\left\lVert#1\right\rVert} 
\newcommand{\iprod}[1]{\left\langle#1\right\rangle} 
\newcommand{\closure}{\operatorname{cl}} 
\newcommand{\Exp}{\operatorname{Exp}} 
\newcommand{\Log}{\operatorname{Log}} 

\begin{document}


\baselineskip=17pt



\title{The horofunction boundary of finite-dimensional $\ell_p$ spaces}

\author[A. W. Guti{é}rrez]{Armando W. Guti{é}rrez}
\address{Department of Mathematics and Systems Analysis\\ Aalto University\\ Otakaari 1 Espoo, Finland}
\email{wladimir.gutierrez@aalto.fi}

\date{}

\begin{abstract}
	We give a complete description of the horofunction boundary of finite-dimensional 
	$\ell_p$ spaces, for $1\leq p\leq \infty$. We also study the variation norm 
	on $\R^{\mathcal{N}}$, $\mathcal{N}=\{1,...,N\}$, 
	and the corresponding horofunction boundary. As a consequence, 
	we describe the horofunctions for Hilbert's projective metric on 
	the interior of the standard cone $\R^{\mathcal{N}}_{+}$ of $\R^{\mathcal{N}}$.
\end{abstract}

\subjclass[2010]{Primary 51F99, 51B20, 52A21; Secondary 46B20}

\keywords{horofunction boundary, metric spaces, $\ell_p$ norms, variation norm, Hilbert's projective metric}

\maketitle

\section{Introduction}
There has recently been growing interest in the horofunction boundary of 
metric spaces. It is a powerful tool in the study of self-mappings of convex 
cones \cite{Gaubert_Vigeral2012,Karlsson2014} and random walks on groups 
\cite{Karlsson_Ledrappier2011}. The horofunction boundary has been studied 
mainly in spaces of nonpositive curvature since the introduction of the 
notion by Gromov \cite{Gromov1981}. By applying methods of convex analysis, Walsh \cite{Walsh2007} 
describes the horofunctions of general finite-dimensional normed spaces. 
Afterwards, in \cite{Walsh2008} he gives a description of the horofunction boundary 
of Hilbert's projective metric on general finite-dimensional cones.
In an earlier paper \cite{Karlsson_Metz_Noskov2006} polyhedral normed spaces and 
Hilbert's projective metric on simplicial cones were studied.

Let $1\leq p\leq \infty$ and let $\mathcal{N}=\{1,...,N\}$ for any $N\in\N$. 
Throughout, we shall denote by $\ell_p(\mathcal{N},\R)$ the vector 
space $\R^{\mathcal{N}}$ endowed with the norm
\[
\norm{x}_{p}=\begin{cases}
					\left( \sum_{i\in\mathcal{N}}\abs{x_i}^{p} \right)^{1/p}, & 1\leq p < \infty, \\
					\max_{i\in\mathcal{N}}\abs{x_i}, & p=\infty,
				\end{cases}
\]
for all $x=(x_i)_{i\in\mathcal{N}}\in\R^{\mathcal{N}}$. We shall also denote by 
$\ell_{\var}(\mathcal{N},\R)$ the vector space $\R^{\mathcal{N}}$ endowed with 
the pseudo-norm 
\[
	\norm{x}_\var=\max_{i\in\mathcal{N}}x_i - \min_{i\in\mathcal{N}}x_i.
\]

The purpose of this paper is to give an explicit and detailed description of
the horofunction boundary of $\ell_p(\mathcal{N},\R)$, for all $1\leq p\leq \infty$. 
We also give a complete description of the horofunction boundary of the pseudo-normed space 
$\ell_{\var}(\mathcal{N},\R)$. As a consequence, we readily obtain the horofunctions for
Hilbert's projective metric on the interior of the standard cone $\R^{\mathcal{N}}_{+}$ 
of $\R^{\mathcal{N}}$. 

We would like to emphasize that the techniques we use in this paper are 
significantly different from those used by Walsh. Our results contain explicit 
formulas for the horofunctions. This paper is organized as follows.
In \Cref{horol1N} we give a complete description of the horofunctions on 
$\ell_1(\mathcal{N},\R)$. In \Cref{horolpN} we show that if $1 < p < \infty$ then the 
horofunction boundary of $\ell_p(\mathcal{N},\R)$ is precisely the set of all norm one 
linear functionals on $\ell_p(\mathcal{N},\R)$. In \Cref{horolinfN} we give a complete 
description of the horofunctions on $\ell_{\infty}(\mathcal{N},\R)$. In \Cref{horolvarN} 
we give a complete description of the horofunction boundary of $\ell_{\var}(\mathcal{N},\R)$, 
and consequently we obtain all the horofunctions for Hilbert's projective metric on 
the interior of the standard cone $\R^{\mathcal{N}}_{+}$ of $\R^{\mathcal{N}}$. 
As an application of the latter result, we give a new proof of Perron's theorem.

\section{Preliminaries}
\subsection{The horofunction boundary of a metric space}
Let $(X,d)$ be a metric space. Fix an arbitrary \emph{base point} $b$ in $X$.
Define the mapping $\tau_d: X \to \R^{X}$ by associating to any $y\in X$
the function $\tau_d(y)$ given by
\begin{equation}\label{emb}
	\tau_d(y)(x):=d(x,y)-d(b,y)
\end{equation}
for all $x$ in $X$. For each $y\in X$, the function $\tau_d(y)$ is bounded from 
below by $-d(b,y)$ and, moreover, is $1$-Lipschitz with respect to the metric $d$. 
In fact, by the triangle inequality it follows that  
\begin{align*}
	\abs{\tau_d(y)(x)-\tau_d(y)(z)} &=\abs{d(x,y)-d(b,y)-d(z,y)+d(b,y)} \\
			&=\abs{d(x,y)-d(z,y)} \\
			&\leq d(x,z)
\end{align*}
for all $x,z\in X$. Furthermore, by taking $z=b$ we get $\abs{\tau_d(y)(x)}\leq d(x,b)$
for all $x\in X$. Hence
\[
	\tau_d(X)\subset \prod_{x\in X}[-d(x,b),d(x,b)]\subset \R^{X}.
\]
By Tychonoff's theorem the product space $\prod_{x\in X}[-d(x,b),d(x,b)]$
is compact in the product topology. Therefore the set $\tau_d(X)$
has compact closure in this topology, which is equivalent to the
topology of pointwise convergence. 

\begin{defin}
We denote by $\overline{X}^{H}:=\closure(\tau_d(X))$ the \emph{horofunction compactification} 
of $(X,d)$. The \emph{horofunction boundary} of $(X,d)$ is defined by 
\begin{equation}\label{horobnd}
	\partial_H X:=\overline{X}^{H}\setminus\tau_d(X).
\end{equation}
The elements of $\partial_H X$ are called \emph{horofunctions} for the metric $d$ on $X$. 
For each $r\in\R$, the sublevel set $\mathcal{H}(h,r):=\{x\in X \mid h(x)\leq r\}$ is called 
a \emph{horoball} centered at $h\in\partial_H X$.
\end{defin}
\begin{rem}
The mapping $y\mapsto\tau_d(y)$ is injective and continuous in the product topology. 
If $(X,d)$ is \emph{proper}, i.e., every closed ball is compact, then the 
mapping $y\mapsto\tau_d(y)$ defines an embedding $X\hookrightarrow\overline{X}^{H}$. 
By identifying $X$ with $\tau_d(X)$, the horofunction boundary (\ref{horobnd}) 
becomes $\partial_H X=\overline{X}^{H}\setminus X$. The choice of the base 
point $b\in X$ is irrelevant, in the sense that horofunction boundaries of $(X,d)$ 
for different base points are homeomorphic. 
We refer to \cite{Ballmann_Gromov_Schroeder1985,Bridson_Haefliger1999,Rieffel2002} 
for further details.
\end{rem}
\begin{rem}
If $X$ is a normed space with norm $\norm{\cdot}$, 
then we choose the base point $b=0\in X$, and so (\ref{emb}) becomes 
$\tau(y)(x)=\norm{x-y}-\norm{y}$. Moreover, if $X$ is finite-dimensional, 
then $(X,\norm{\cdot})$ is a proper metric space and hence any $h\in\overline{X}^{H}$
can be written as $h(x)=\lim_{n\to\infty}\tau(y^{n})(x)$, for all $x\in X$ and 
for some sequence $\{y^{n}\}_{n\in\N}$ in $X$. 
\end{rem}

It is well-known that the horofunction boundary of $\ell_1(\{1\},\R):=(\R,\abs{\cdot})$ 
has exactly two elements. More precisely, by considering unbounded sequences 
$\{y^{n}\}_{n\in\N}$ of real numbers, one obtains 
\begin{equation}\label{l11horo}
	\partial_H \ell_1(\{1\},\R)= \left\lbrace 
						x\mapsto h_{\epsilon} (x)=\epsilon x
							\mathrel{\big\vert} 
							\begin{aligned}							
								\epsilon\in\{-1,+1\}
							 \end{aligned} 	
						\right\rbrace.
\end{equation}
In the following sections we describe the horofunction boundary of $\ell_p(\mathcal{N},\R)$ 
for all $1 \leq p \leq \infty$. Throughout we shall denote 
$\tau_p(y)(x)=\norm{x-y}_p-\norm{y}_p$, where $x,y\in\ell_p(\mathcal{N},\R)$.

\section{The horofunction boundary of $\ell_1(\mathcal{N},\R)$}\label{horol1N}
\begin{lem}\label{ThmAl1}
Let $N\geq 2$ and $\mathcal{N}=\{1,...,N\}$. Let $\{y^{n}\}_{n\in\N}$ be a 
sequence in $\ell_1(\mathcal{N},\R)$ such that $\norm{y^{n}}_1 \to \infty$ 
as $n \to \infty$. Then there exists $\emptyset\subsetneq\mathcal{I}\subseteq\mathcal{N}$
such that the sequence of functions $\{\tau_{1}(y^{n})\}_{n\in\N}$ has a
subsequence which converges pointwise to the function
\[
	x\mapsto h_{\epsilon,\mu}^{\mathcal{I}}(x):=\sum_{i\in\mathcal{I}}\epsilon_ix_i + 
		\sum_{i\in\mathcal{N}\setminus\mathcal{I}}(\abs{x_i-\mu_i}-\abs{\mu_i}),
\]
where $\epsilon=(\epsilon_i)_i\in\{-1,+1\}^{\mathcal{I}}$ and $\mu=(\mu_i)_i\in\R^{\mathcal{N}\setminus\mathcal{I}}$.
\end{lem}
\begin{proof}
Let $\{y^{n}\}_{n\in\N}$ be a sequence in $\ell_1(\mathcal{N},\R)$ such that 
$\norm{y^{n}}_1 \to \infty$, as $n \to \infty$.
By taking subsequences, we can find $\emptyset\subsetneq\mathcal{I}\subseteq\mathcal{N}$ such that
$\abs{y_i^{n}}\to  \infty$, as $n \to \infty$, for all $i\in\mathcal{I}$, and
$\{y_i^{n}\}_{n\in\N}\subset\R$ is bounded for all $i\in\mathcal{N}\setminus\mathcal{I}$.
By applying Cantor's diagonal argument and (\ref{l11horo}), we find a
further subsequence such that for every $x\in\ell_1(\mathcal{N},\R)$,
\begin{align*}
	\tau_{1}(y^{n})(x) &=\sum_{i\in\mathcal{N}}\abs{x_i-y_i^{n}} - \sum_{i\in\mathcal{N}}\abs{y_i^{n}} \\
						&= \sum_{i\in\mathcal{I}} \left( \abs{x_i-y_i^{n}}-\abs{y_i^{n}} \right)  +
						\sum_{i\in\mathcal{N}\setminus\mathcal{I}} \left( \abs{x_i-y_i^{n}}-\abs{y_i^{n}} \right)\\
						&\xrightarrow[n \to \infty]{} \sum_{i\in\mathcal{I}}\epsilon_ix_i + 
						\sum_{i\in\mathcal{N}\setminus\mathcal{I}}(\abs{x_i-\mu_i}-\abs{\mu_i}),
\end{align*}
where $\epsilon=(\epsilon_i)_i\in\{-1,+1\}^{\mathcal{I}}$ and $\mu=(\mu_i)_i\in\R^{\mathcal{N}\setminus\mathcal{I}}$.
\end{proof}
\begin{thm}\label{ThmBl1}
Let $N\geq 2$ and $\mathcal{N}=\{1,...,N\}$. The horofunction boundary of the metric 
space $\ell_1(\mathcal{N},\R)$ is given by 
\begin{equation}\label{l1horoset}
	\partial_H \ell_1(\mathcal{N},\R)= \left\lbrace 
						x\mapsto h_{\epsilon,\mu}^{\mathcal{I}}(x)
							\mathrel{\bigg\vert} 
							  \begin{aligned}						
								&\emptyset\subsetneq\mathcal{I}\subseteq\mathcal{N},\;
								\epsilon\in\{-1,+1\}^{\mathcal{I}},\\ 
								&\mu\in\R^{\mathcal{N}\setminus\mathcal{I}}	
							  \end{aligned}	
						\right\rbrace,
\end{equation}
where $h_{\epsilon,\mu}^{\mathcal{I}}(x)=\sum_{i\in\mathcal{I}}\epsilon_ix_i + 
		\sum_{i\in\mathcal{N}\setminus\mathcal{I}}(\abs{x_i-\mu_i}-\abs{\mu_i})$ 
		for all $x\in\ell_1(\mathcal{N},\R)$.
\end{thm}
\begin{proof}
Suppose that $h\in\partial_H \ell_1(\mathcal{N},\R)$. Then there exists a sequence $\{y^{n}\}_{n\in\N}$ in
$\ell_1(\mathcal{N},\R)$ with $\norm{y^{n}}_1 \to \infty$ such that $\{\tau_{1}(y^{n})\}_{n\in\N}$ converges pointwise
to $h$ as $n \to \infty$.
By \Cref{ThmAl1} there exist
$\emptyset\subsetneq\mathcal{I}\subseteq\mathcal{N}$,
$\epsilon\in\{-1,+1\}^{\mathcal{I}}$ and $\mu\in\R^{\mathcal{N}\setminus\mathcal{I}}$
such that there is a subsequence $\{\tau_{1}(y^{n_k})\}_{k}$ that converges pointwise to $ h_{\epsilon,\mu}^{\mathcal{I}}$ as 
$k \to \infty $. Therefore $h=h_{\epsilon,\mu}^{\mathcal{I}}$ and so $\partial_H \ell_1(\mathcal{N},\R)$ is contained 
in the set on the right-hand side of (\ref{l1horoset}).

For the other inclusion, assume that $\mathcal{I}$ is any nonempty subset of $\mathcal{N}$. 
Let $\epsilon\in\{-1,+1\}^{\mathcal{I}}$ and let $\mu\in\R^{\mathcal{N}\setminus\mathcal{I}}$.
We will show that the function $h_{\epsilon,\mu}^{\mathcal{I}}$ belongs to 
$\overline{\ell_1(\mathcal{N},\R)}^{H}\setminus\tau_1(\ell_1(\mathcal{N},\R))$.
Indeed, for each $n$ define $y^{n}=(y_{i}^{n})_{i\in\mathcal{N}}$ in 
$\ell_1(\mathcal{N},\R)$ by
\begin{equation}\label{escapevector}
	y_i^{n}=\begin{cases} -\epsilon_i n, & i\in\mathcal{I}, \\
							\mu_i,  & i\in\mathcal{N}\setminus\mathcal{I}.
				\end{cases}		
\end{equation}
Then for every $x\in\ell_1(\mathcal{N},\R)$ we have
\begin{align*}
	\tau_{1}(y^{n})(x) &=\sum_{i\in\mathcal{I}} \left( \abs{x_i-y_i^{n}}-\abs{y_i^{n}} \right)  +
						\sum_{i\in\mathcal{N}\setminus\mathcal{I}} \left( \abs{x_i-y_i^{n}}-\abs{y_i^{n}} \right)\\
						&=\sum_{i\in\mathcal{I}} \left( \abs{x_i+\epsilon_i n}-n \right)  +
						\sum_{i\in\mathcal{N}\setminus\mathcal{I}} \left( \abs{x_i-\mu_i}-\abs{\mu_i} \right)\\
						&\xrightarrow[n \to \infty]{} \sum_{i\in\mathcal{I}}\epsilon_ix_i + 
						\sum_{i\in\mathcal{N}\setminus\mathcal{I}} \left( \abs{x_i-\mu_i}-\abs{\mu_i} \right) 
						= h_{\epsilon,\mu}^{\mathcal{I}}(x).
\end{align*}
Therefore $h_{\epsilon,\mu}^{\mathcal{I}}\in\overline{\ell_1(\mathcal{N},\R)}^{H}$.
It remains to show that $h_{\epsilon,\mu}^{\mathcal{I}}$ is not an element of $\tau_1(\ell_1(\mathcal{N},\R))$.
Suppose the contrary, so there exists $z\in\ell_1(\mathcal{N},\R)$ such that $h_{\epsilon,\mu}^{\mathcal{I}}=\tau_{1}(z)$. 
It follows by (\ref{escapevector}) that 
\[
	h_{\epsilon,\mu}^{\mathcal{I}}(y^{n})=-n\abs{\mathcal{I}} - \sum_{i\in\mathcal{N}\setminus\mathcal{I}}\abs{\mu_i}\xrightarrow[n \to \infty]{}-\infty.
\]
However, by (\ref{emb}) we know that $\tau_{1}(z)$ is bounded from below by $-\norm{z}_1$, and hence
\[
	\liminf_{n\to\infty}\tau_{1}(z)(y^{n})\geq -\norm{z}_1 >  -\infty,
\]
which is a contradiction. Therefore $h_{\epsilon,\mu}^{\mathcal{I}}$ belongs to $\partial_H \ell_1(\mathcal{N},\R)$, that is,
every element of the set on the right-hand side of (\ref{l1horoset}) is a horofunction on $\ell_1(\mathcal{N},\R)$.
\end{proof}

\section{The horofunction boundary of $\ell_p(\mathcal{N},\R)$ for $1<p<\infty$}\label{horolpN}
Recall that a normed space $(X,\norm{\cdot})$ is called \textit{uniformly convex} 
if for every $\epsilon\in]0,2]$ there exists $\delta>0$ such that
$\norm{x+y}\leq 2(1-\delta)$ whenever $x,y\in X$ with $\norm{x}=\norm{y}=1$
and $\norm{x-y}\geq\epsilon$. A well-known result due to Clarkson \cite{Clarkson1936} 
is that $L_{p}$ and $\ell_p$ spaces are uniformly convex for $1<p<\infty$. It will be convenient
to use the following equivalent characterization of uniform convexity.
\begin{prop}({\cite[p.~287]{INF_DIM_GEOM2001}})\label{uniconvex}
A Banach space $(X,\norm{\cdot})$ is uniformly convex if and only if 
$\norm{x_n-y_n}\to 0$, as $n\to \infty$, whenever $x_n,y_n\in X$ 
with $\norm{x_n}\leq 1,\norm{y_n}\leq 1$ for all $n\in\N$, and $\norm{x_n+y_n}\to 2$ as $n\to \infty$.
\end{prop}

\begin{lem}\label{ThmAlp}
Let $p,q\in ]1,+\infty[$ such that $p^{-1}+q^{-1}=1$. Let $\{y^{n}\}_{n\in\N}$ 
be a sequence in $\ell_p(\mathcal{N},\R)$ 
such that $\norm{y^{n}}_p \to \infty$ as $n\to \infty$. Then there exists 
$\mu\in \ell_q(\mathcal{N},\R)$ with $\norm{\mu}_q=1$ for which the sequence of 
functions $\{\tau_{p}(y^{n})\}_{n\in\N}$ has a
subsequence converging pointwise to the function
\[
	x\mapsto h_{\mu}(x):=-\sum_{i\in\mathcal{N}}\mu_ix_i.
\]
\end{lem}
\begin{proof}
We may assume without loss of generality that $y^{n}\neq 0$, and define 
$w^{n}:=y^{n}/\norm{y^{n}}_p$ for all $n$. By compactness of the unit sphere of 
$\ell_p(\mathcal{N},\R)$, it follows that there exists a subsequence $\{w^{n_k}\}_k$
that converges, as $k\to\infty$, to some $w\in\ell_p(\mathcal{N},\R)$ with $\norm{w}_p=1$.
Therefore, by $\ell_p/\ell_q$-duality there exists a unique $\mu\in\ell_q(\mathcal{N},\R)$
with $\norm{\mu}_q=1$ such that $\iprod{\mu,w}=1$. 
Now, let $x\in\ell_p(\mathcal{N},\R)$ and for each $k$ define
\begin{equation}\label{unitzk}
	z^{k}:=\frac{y^{n_k}-x}{\norm{x-y^{n_k}}_p}
		=\frac{-x}{\norm{x-y^{n_k}}_p}+\frac{\norm{y^{n_k}}_p}{\norm{x-y^{n_k}}_p}w^{n_k}.
\end{equation}
For each $k$ we have $\norm{z^{k}}_p=1$, and hence by $\ell_p/\ell_q$-duality 
there exists $\varphi^{k}\in\ell_q(\mathcal{N},\R)$ 
with $\norm{\varphi^{k}}_q=1$ such that $\iprod{\varphi^{k},z^{k}}=1$. 
By applying the assumption $\norm{y^{n_k}}_p \to \infty$ to (\ref{unitzk}), we obtain
$\norm{z^{k}-w}_p\to 0$ as $k\to\infty$. Consequently,
\[
	2=\norm{\varphi^{k}}_q+\norm{\mu}_q\geq\norm{\varphi^{k}+\mu}_q\geq \iprod{\varphi^{k}+\mu,z^{k}}=1+\iprod{\mu,z^{k}}
	\xrightarrow[k \to \infty]{}2,
\]
and hence, by \Cref{uniconvex}, we have $\norm{\varphi^{k}-\mu}_q\to 0$ as $k\to\infty$. 
On the other hand, by evaluating each dual pairing of $\varphi^{k}$ at $z^{k}$ in (\ref{unitzk}) we obtain
\begin{equation*}
	\norm{x-y^{n_k}}_p=\iprod{\varphi^{k},-x}+\norm{y^{n_k}}_p\iprod{\varphi^{k},w^{n_k}}.
\end{equation*}
Therefore,
\begin{align*}
	\tau_{p}(y^{n_k})(x)&=\norm{x-y^{n_k}}_p-\norm{y^{n_k}}_p\\
						&=\iprod{\varphi^{k},-x}+\norm{y^{n_k}}_p\iprod{\varphi^{k},w^{n_k}}-\norm{y^{n_k}}_p\\
						&\xrightarrow[k \to \infty]{} -\iprod{\mu,x}=h_{\mu}(x).
\end{align*}
\end{proof}

\begin{thm}\label{ThmBlp}
Let $p,q\in ]1,+\infty[$ such that $p^{-1}+q^{-1}=1$. The horofunction boundary of the metric space $\ell_p(\mathcal{N},\R)$ 
is given by 
\begin{equation}\label{lphoroset}
	\partial_H\ell_p(\mathcal{N},\R)= \left\lbrace 
						x\mapsto h_{\mu}(x)
							\mathrel{\big\vert} 						
								\mu\in\ell_q(\mathcal{N},\R),\; \norm{\mu}_q=1
						\right\rbrace,
\end{equation}
where $h_{\mu}(x)=-\sum_{i\in\mathcal{N}}\mu_ix_i$ for all $x\in\ell_p(\mathcal{N},\R)$.
\end{thm}
\begin{proof}
If $h\in\partial_H\ell_p(\mathcal{N},\R)$, then there exists a sequence $\{y^{n}\}_{n\in\N}$ 
with $\norm{y^{n}}_p \to \infty$ such that $h$ is the pointwise limit of the sequence 
$\{\tau_{p}(y^{n})\}_{n\in\N}$. By \Cref{ThmAlp}, there exists $\mu\in\ell_q(\mathcal{N},\R)$ with $\norm{\mu}_q=1$ such that 
along subsequences $\tau_{p}(y^{n})(x)$ 
converges to $h_{\mu}(x)=-\sum_{i\in\mathcal{N}}\mu_ix_i$ for all $x\in\ell_p(\mathcal{N},\R)$. Therefore $h=h_\mu$ and so
$\partial_H \ell_p(\mathcal{N},\R)$ is contained 
in the set on the right-hand side of (\ref{lphoroset}).

On the other hand, if $\mu\in\ell_q(\mathcal{N},\R)$ with $\norm{\mu}_q=1$, then by $\ell_p/\ell_q$-duality 
there exists $w\in\ell_p(\mathcal{N},\R)$ 
with $\norm{w}_p=1$ such that $\sum_{i\in\mathcal{N}}\mu_i w_i=1$. Let $y^{n}=nw$ for all $n$. 
Then, by proceeding as in \Cref{ThmAlp}, it
follows that $\tau_{p}(y^{n})(x)=\norm{x-nw}_p-n$ converges to $h_{\mu}(x)=-\sum_{i\in\mathcal{N}}\mu_ix_i$
for all $x\in\ell_p(\mathcal{N},\R)$. That is, $h_{\mu}$ belongs to $\overline{\ell_p(\mathcal{N},\R)}^{H}$. However,
note that $h_{\mu}(y^{n})=-n$ for all $n$. Therefore, since for any $z\in\ell_p(\mathcal{N},\R)$ 
the function $\tau_{p}(z)$ is bounded from below, we must have $h_{\mu}\in\partial_H \ell_p(\mathcal{N},\R)$.
That is, every element of the set on the right-hand side of (\ref{lphoroset}) is a horofunction on $\ell_p(\mathcal{N},\R)$.  
\end{proof}
\begin{rem}
\Cref{ThmAlp} and \Cref{ThmBlp} hold for every finite-dimensional uniformly convex Banach space.
\end{rem}

\section{The horofunction boundary of $\ell_\infty(\mathcal{N},\R)$}\label{horolinfN}
It will be convenient and helpful to consider the \emph{top function} $\topf$ and 
the \emph{bottom function} $\botf$ defined on $\R^{\mathcal{N}}$ by 
\begin{equation*}
	\begin{aligned}
		\topf(x) &:= \max_{i\in\mathcal{N}}x_i,\;\;
	\end{aligned}
	\begin{aligned}
		\botf(x) &:= \min_{i\in\mathcal{N}}x_i. 
	\end{aligned}	
\end{equation*}
These functions simplify notations significantly 
when proving \Cref{ThmAlsup} and \Cref{ThmBlsup} in this section as well as 
\Cref{ThmAlvar}, \Cref{ThmBlvar}, and \Cref{horoHilbert} in \Cref{horolvarN}. 
The norm $\norm{\cdot}_\infty$ on $\R^{\mathcal{N}}$ can 
be redefined as 
\begin{equation}\label{maxnorm}
   \norm{x}_\infty=\max\{\topf(x),-\botf(x)\}.
\end{equation}

The standard cone $\R_{+}^{\mathcal{N}}$ of $\R^{\mathcal{N}}$ is defined by
\[
\R_{+}^{\mathcal{N}}:=\{x\in\R^{\mathcal{N}} \mid x_i\geq 0,\;\forall i\in\mathcal{N}\}.
\]
We denote by $\R_{>0}^{\mathcal{N}}$ the interior of $\R_{+}^{\mathcal{N}}$. The boundary $\partial\R_{+}^{\mathcal{N}}$ of 
$\R_{+}^{\mathcal{N}}$ is the set $\R_{+}^{\mathcal{N}}\setminus\R_{>0}^{\mathcal{N}}$. We shall denote
by $\cht$ the element of $\R^{\mathcal{N}}$ given by $\cht=(1,...,1)$.
It follows that, $x - \botf(x)\cht$ and $\topf(x)\cht - x$ are both elements of $\partial\R_{+}^{\mathcal{N}}$ 
for all $x\in\R^{\mathcal{N}}$.
The mapping $\Exp : \R^{\mathcal{N}}\to \R_{>0}^{\mathcal{N}}$ is defined by $\Exp(x)_i := e^{x_i}$
for all $i\in\mathcal{N}$. Similarly, the mapping $\Log : \R_{>0}^{\mathcal{N}}\to\R^{\mathcal{N}}$ is defined by
$\Log(u)_i := \log(u_i)$ for all $i\in\mathcal{N}$. 

The \textit{Hadamard product} of any two elements $x=(x_i)_{i\in\mathcal{N}}$ and 
$y=(y_i)_{i\in\mathcal{N}}$ of $\R^{\mathcal{N}}$, denoted by $x \odot y$, is another element of 
$\R^{\mathcal{N}}$ defined by $(x \odot y)_i := x_iy_i$ for all $i\in\mathcal{N}$. 
For every $x=(x_i)_{i\in\mathcal{N}}$ in $\R_{>0}^{\mathcal{N}}$ we 
shall denote by $x^{-1}$ the element of $\R_{>0}^{\mathcal{N}}$ defined by $(x^{-1})_i:=1/x_i$ for all $i\in\mathcal{N}$.

Using the notations introduced above, it readily follows that
\begin{align}
	\topf(\Exp(x)\odot\Exp(y)) &= \exp \topf(x+y) \mbox{ for all } x,y\in \R^{\mathcal{N}}, \label{t-ExpExp} \\
	\topf(\Log(x)-\Log(y)) &= \log \topf(x\odot y^{-1}) \mbox{ for all } x,y\in\R_{>0}^{\mathcal{N}}. \label{t-Log-Log}
\end{align}

Note that $\norm{x}_\infty=\max\{\topf(x),\topf(-x)\}=\topf(x,-x)$ for all $x\in\R^{\mathcal{N}}$. 
Therefore, the mapping $y\mapsto\tau_{\infty}(y)$ becomes
\begin{align*}
	\tau_{\infty}(y)(x) &=\norm{x-y}_\infty-\norm{y}_\infty \\
						&=\topf(x-y,-x+y)-\norm{y}_\infty \\
						&=\topf(x-y-\norm{y}_\infty\cht,-x+y-\norm{y}_\infty\cht).
\end{align*}
For any $x=(x_i)_{i\in\mathcal{N}}\in\R^{\mathcal{N}}$ and any nonempty subset $\mathcal{I}$ of $\mathcal{N}$ we shall
denote $x_\mathcal{I}=(x_i)_{i\in\mathcal{I}}$. 

\begin{lem}\label{ThmAlsup}
Let $\{y^{n}\}_{n\in\N}$ be a sequence in $\ell_\infty(\mathcal{N},\R)$ 
such that $\norm{y^{n}}_\infty \to \infty$, as $n\to \infty$. Then the sequence $\{\tau_{\infty}(y^{n})\}_{n\in\N}$ 
has a subsequence which converges pointwise to the function
\[
	x\mapsto h^{\mathcal{I},\mathcal{J}}_{\mu,\nu}(x):=\topf(x_\mathcal{I}-\mu,-x_\mathcal{J}-\nu),
\]
where $\emptyset\subseteq\mathcal{I},\mathcal{J} \subseteq \mathcal{N}$
with $\mathcal{I}\cap\mathcal{J}=\emptyset$, $\mathcal{I}\cup\mathcal{J}\neq\emptyset$,
and $\mu\in\R_{+}^{\mathcal{I}},\,\nu\in \R_{+}^{\mathcal{J}}$ with $\botf(\mu,\nu)=0$.
\end{lem}
\begin{proof}
For each $n$, define 
\begin{align*}
	u^{n}&=\Exp(-y^{n}-\norm{y^{n}}_\infty\cht),\\
	v^{n}&=\Exp(y^{n}-\norm{y^{n}}_\infty\cht).
\end{align*} 
It follows that $(u^{n},v^{n})\in \R_{+}^{\mathcal{N}}\times \R_{+}^{\mathcal{N}}$ 
with $\topf(u^{n},v^{n})=1$ for all $n$. 
Therefore, there exists a subsequence $\{(u^{n_k},v^{n_k})\}_{k}$ 
which converges, as $k\to\infty$, 
to $(u,v)\in \R_{+}^{\mathcal{N}}\times \R_{+}^{\mathcal{N}}$ 
with $\topf(u,v)=1$. 
Furthermore, note that 
$u^{n_k}\odot v^{n_k}=\Exp(-2\norm{y^{n_k}}_\infty\cht)$ for all $k$. 
Hence by taking the limit as $k\to\infty$, we obtain $u\odot v =0\cht$. 
Consequently, there exist 
$\emptyset\subseteq\mathcal{I},\mathcal{J}\subseteq\mathcal{N}$ 
with $\mathcal{I}\cap\mathcal{J}=\emptyset$
and $\mathcal{I}\cup\mathcal{J}\neq\emptyset$
such that $0<u_i\leq 1$ for all $i\in\mathcal{I}$, 
$0< v_j \leq 1$ for all $j\in\mathcal{J}$ 
with $\topf(u_\mathcal{I},v_\mathcal{J})=1$. 
Now, by letting $\mu=-\Log(u_\mathcal{I})$ and $\nu=-\Log(v_\mathcal{J})$, 
it follows that $\mu\in\R_{+}^{\mathcal{I}},\nu\in \R_{+}^{\mathcal{J}}$ 
with $\botf(\mu,\nu)=0$.
Finally, let $x\in\ell_\infty(\mathcal{N},\R)$; then by (\ref{t-ExpExp}) 
we have
\begin{align*}
\lim_{k\to\infty} \tau_{\infty}(y^{n_{k}})(x) 
			&=\lim_{k\to\infty}\topf(x-y^{n_k}-\norm{y^{n_k}}_\infty\cht,-x+y^{n_k}-\norm{y^{n_k}}_\infty\cht)\\
			&=\lim_{k\to\infty}\log\topf(\Exp(x)\odot u^{n_k},\Exp(-x)\odot v^{n_k})\\
			&=\log\topf(\Exp(x)\odot u,\Exp(-x)\odot v)\\
			&=\log\topf(\Exp(x_\mathcal{I})\odot u_\mathcal{I},\Exp(-x_\mathcal{J})\odot v_\mathcal{J})\\
			&=\topf(x_\mathcal{I}-\mu,-x_\mathcal{J}-\nu).
\end{align*}
\end{proof}

Let $\overline\R$ denote the extended set of 
real numbers $\R\cup\{-\infty,\infty\}$. 
The top function $\topf$ and bottom function $\botf$ 
can be redefined on $\overline\R^{\mathcal{N}}$
according to the natural order in $\overline\R$. 
Let $\overline\R_+^{\mathcal{N}}$ denote the set 
\[
	\overline\R_+^{\mathcal{N}}=
		\{x\in\overline\R^{\mathcal{N}}\mid 0\leq x_i\leq \infty,\;\forall i=1,...,N\}
		=[0,\infty]^{\mathcal{N}}.
\] 
\begin{thm}\label{ThmBlsup}
The horofunction boundary of the metric space 
$\ell_\infty(\mathcal{N},\R)$ is given by
\begin{equation}\label{linfhoroset}
	\partial_H\ell_\infty(\mathcal{N},\R)= \left\lbrace 
						x\mapsto h_{\overline\mu,\overline\nu}(x)
							\mathrel{\bigg\vert} 
							\begin{aligned}							
									&\overline\mu,\overline\nu\in\overline\R_+^{\mathcal{N}},\;
									\botf(\overline\mu,\overline\nu)=0,\;\\
									&\overline\mu+\overline\nu=\infty\cht
							 \end{aligned} 					
						\right\rbrace,
\end{equation}
where $h_{\overline\mu,\overline\nu}(x)=\topf(x-\overline\mu,-x-\overline\nu)$
for all $x\in\ell_\infty(\mathcal{N},\R)$.
\end{thm}
\begin{proof}
Suppose that $h\in\partial_H\ell_\infty(\mathcal{N},\R)$. 
Then there exists a sequence $\{y^{n}\}_{n\in\N}$ in 
$\ell_\infty(\mathcal{N},\R)$ with $\norm{y^{n}}_\infty \to \infty$ 
such that $\tau_{\infty}(y^{n})$
converges pointwise to $h$ as $n \to \infty$. 
Let $x\in\ell_\infty(\mathcal{N},\R)$. 
By \Cref{ThmAlsup} there exist
$\emptyset\subseteq\mathcal{I},\mathcal{J}\subseteq\mathcal{N}$
with $\mathcal{I}\cap\mathcal{J}=\emptyset$, $\mathcal{I}\cup\mathcal{J}\neq\emptyset$,
and $\mu\in\R_{+}^{\mathcal{I}},\nu\in \R_{+}^{\mathcal{J}}$ 
with $\botf(\mu,\nu)=0$,
such that for some subsequence $\{y^{n_k}\}_k$ we have
\[
	h(x)=\lim_{k\to\infty}\tau_{\infty}(y^{n_k})(x)
		=\topf(x_\mathcal{I}-\mu,-x_\mathcal{J}-\nu).
\]
Finally, by letting 
\begin{equation*}
\begin{aligned}
	\overline\mu_{i}=\begin{cases} \mu_i, & i\in\mathcal{I} \\
							\infty  & i\in\mathcal{N}\setminus\mathcal{I}
				\end{cases}		
\end{aligned}
\begin{aligned}
\;,\;\;
\end{aligned}
\begin{aligned}
	\overline\nu_{i}=\begin{cases} \nu_i, & i\in\mathcal{J} \\
							\infty  & i\in\mathcal{N}\setminus\mathcal{J}
				\end{cases}		
\end{aligned}
\end{equation*}
we get $\overline\mu,\overline\nu\in\overline\R_+^{\mathcal{N}}$ 
with $\overline\mu+\overline\nu=\infty\cht$,
and $\botf(\overline\mu,\overline\nu)=0$. 
Hence, $h(x)=h_{\overline\mu,\overline\nu}(x)$ 
and so $\partial_H\ell_\infty(\mathcal{N},\R)$ is contained 
in the set on the right-hand side of (\ref{linfhoroset}).

Now, we need to show that given 
$\overline\mu,\overline\nu\in\overline\R_+^{\mathcal{N}}$ 
with $\overline\mu+\overline\nu=\infty\cht$
and $\botf(\overline\mu,\overline\nu)=0$,
the function $x\mapsto h_{\overline\mu,\overline\nu}(x)$ 
is a horofunction on $\ell_\infty(\mathcal{N},\R)$.
First we show that it belongs to $\overline{\ell_\infty(\mathcal{N},\R)}^{H}$.
Indeed, let $(y^{n})_n$ be the sequence 
in $\ell_\infty(\mathcal{N},\R)$ given by
\[
	y_i^{n}=\begin{cases} -n+\overline\mu_i, &  \overline\mu_i <\infty\\
							n-\overline\nu_i, & \overline\nu_i <\infty\\
							0,  & \mbox{otherwise}.
				\end{cases}		
\]
Let $x\in\ell_\infty(\mathcal{N},\R)$. Then
\[
	\tau_{\infty}(y^{n})(x)=\norm{x-y^{n}}_\infty-\norm{y^{n}}_\infty 
	\xrightarrow[n \to \infty]{} \topf(x-\overline\mu,-x-\overline\nu)
	= h_{\overline\mu,\overline\nu}(x). 
\]
It remains to show that $h_{\overline\mu,\overline\nu}$ is not an
element of $\tau_{\infty}(\ell_\infty(\mathcal{N},\R))$.
Suppose the contrary, so there exists $z\in\ell_\infty(\mathcal{N},\R)$ 
such that
\[
	h_{\overline\mu,\overline\nu}(x)=\topf(x-\overline\mu,-x-\overline\nu)
	=\tau_{\infty}(z)(x).
\]
For each $k$, define
\[
	x_i^{k}=\begin{cases} \overline\mu_i, & \overline\mu_i <\infty \\
							-\overline\nu_i, & \overline\nu_i <\infty\\
							-k,  & \mbox{otherwise}.
				\end{cases}			
\]
Then $h_{\overline\mu,\overline\nu}(x^{k})=\topf(x^{k}-\overline\mu,-x^{k}-\overline\nu)=0$
for all $k$. However,
\[
	\tau_{\infty}(z)(x^{k})=\norm{x^{k}-z}_\infty-\norm{z}_\infty 
	\centernot{\xrightarrow[k \to \infty]{}0},
\]
which is a contradiction. 
Therefore $h_{\overline\mu,\overline\nu}\in\partial_H\ell_\infty(\mathcal{N},\R)$
and so the other inclusion holds.
\end{proof}

\section{The horofunction boundary of $\ell_{\var}(\mathcal{N},\R)$}\label{horolvarN}
We define the \textit{variation norm} 
on $\R^{\mathcal{N}}$ by
\begin{equation}\label{varnorm}
	\norm{x}_\var := \topf(x)-\botf(x),
\end{equation}
where $\topf$ and $\botf$ are, respectively, 
the top and bottom functions introduced in \Cref{horolinfN}.
In fact, $\norm{\cdot}_\var$ is a pseudo-norm 
on $\R^{\mathcal{N}}$, as $\norm{x}_\var = 0$ 
if and only if $x=\lambda\cht$ for some $\lambda\in\R$.
Moreover $\norm{x+\lambda\cht}_\var = \norm{x}_\var$ 
for all $x\in\R^{\mathcal{N}}$.
Hence $\norm{\cdot}_\var$ is a norm on the quotient 
vector space $\R^{\mathcal{N}}/\R\cht$. 
By (\ref{varnorm}), the mapping $y\mapsto\tau_\var(y)$ 
becomes
\begin{align}
		\tau_\var(y)(x) &= \norm{x-y}_\var - \norm{y}_\var \nonumber \\
						  &= \topf(x-y)-\botf(x-y) - \topf(y)+\botf(y) \nonumber \\
						  &= \topf(x-y+\botf(y)\cht) - \botf(x-y+\topf(y)\cht) \label{emblvar}. 
\end{align}
\begin{lem}\label{ThmAlvar}
Let $N\geq 2$ and $\mathcal{N}=\{1,...,N\}$. 
Let $\{y^{n}\}_{n\in\N}$ be a sequence in $\ell_{\var}(\mathcal{N},\R)$ 
such that $\norm{y^{n}}_\var \to \infty$, as $n\to \infty$. 
Then $\{\tau_\var(y^{n})\}_{n\in\N}$ has a 
subsequence which converges pointwise to the function
\[
	x\mapsto h_{\mu,\nu}^{\mathcal{I},\mathcal{J}}(x)
	:=\topf(x_\mathcal{I}-\mu) - \botf(x_\mathcal{J}+\nu),
\]
where $\emptyset\subsetneq\mathcal{I},\mathcal{J}\subsetneq\mathcal{N}$
with $\mathcal{I}\cap\mathcal{J}=\emptyset$, and
$\mu\in\partial\R_{+}^{\mathcal{I}}$, $\nu\in \partial\R_{+}^{\mathcal{J}}$.
\end{lem}
\begin{proof}
For each $n$, define
\begin{align*}
	u^{n} &=\Exp(\botf(y^{n})\cht - y^{n}),\\
	v^{n} &=\Exp(y^{n}-\topf(y^{n})\cht).
\end{align*}
It follows that $(u^{n},v^{n})\in \R_{+}^{\mathcal{N}}\times\R_{+}^{\mathcal{N}}$ 
with $\topf(u^{n})=1$, $\topf(v^{n})=1$ for all $n$. 
Hence, there exists a subsequence
$\{(u^{n_k},v^{n_k})\}_k$ which converges, as $k\to\infty$, 
to some $(u,v)\in \R_{+}^{\mathcal{N}}\times\R_{+}^{\mathcal{N}}$ 
with $\topf(u)=1$, $\topf(v)=1$. 
Let $x\in\ell_{\var}(\mathcal{N},\R)$. 
By (\ref{t-ExpExp}), it follows that for every $k$,
\begin{enumerate}[\upshape (i)]
\item $\log\topf(\Exp(x)\odot u^{n_{k}}) = \topf(x-y^{n_{k}}+\botf(y^{n_{k}})\cht)$,
\item $\log\topf(\Exp(-x)\odot v^{n_{k}})=\topf(-x+y^{n_{k}}-\topf(y^{n_{k}})\cht)=-\botf(x-y^{n_{k}}+\topf(y^{n_{k}})\cht)$.
\end{enumerate}
Therefore, by (\ref{emblvar})
\begin{align*}
	\tau_\var(y^{n_{k}})(x) &=\topf(x-y^{n_{k}}+\botf(y^{n_{k}})\cht) - \botf(x-y^{n_{k}}+\topf(y^{n_{k}})\cht)\\
						  &=\log\topf(\Exp(x)\odot u^{n_{k}}) + \log\topf(\Exp(-x)\odot v^{n_{k}}) \\
						  &\xrightarrow[k \to \infty]{} \log\topf(\Exp(x)\odot u) + \log\topf(\Exp(-x)\odot v).
\end{align*}
Also note that 
$u^{n_{k}} \odot v^{n_{k}} = \exp(-\norm{y^{n_{k}}}_\var)\cht$
for all $k$. Hence, by taking the limit as $k\to\infty$ 
we obtain $u \odot v =0\cht$. Consequently, there exist 
$\emptyset\subsetneq\mathcal{I},\mathcal{J}\subsetneq\mathcal{N}$ 
with $\mathcal{I}\cap\mathcal{J}=\emptyset$  
such that $0<u_i\leq 1$ for all $i\in \mathcal{I}$, 
and $0<v_j\leq 1$ for all $j\in \mathcal{J}$.
Let $\mu=-\Log(u_\mathcal{I})$ and $\nu=-\Log(v_\mathcal{J})$. 
Then $\mu\in\partial\R_{+}^{\mathcal{I}}$ 
and $\nu\in \partial\R_{+}^{\mathcal{J}}$. 
Therefore, by (\ref{t-ExpExp}), it follows that
\begin{align*}
\lim_{k\to\infty} \tau_\var(y^{n_{k}})(x) &= \log\topf(\Exp(x)\odot u) + \log\topf(\Exp(-x)\odot v)\\
							&=\log\topf(\Exp(x_\mathcal{I})\odot u_\mathcal{I}) + \log\topf(\Exp(-x_\mathcal{J})\odot v_\mathcal{J})\\
							&=\topf(x_\mathcal{I}-\mu)+\topf(-x_\mathcal{J}-\nu)\\
							&=\topf(x_\mathcal{I}-\mu)-\botf(x_\mathcal{J}+\nu)\\
							&=h_{\mu,\nu}^{\mathcal{I},\mathcal{J}}(x).
\end{align*}
\end{proof}
\begin{thm}\label{ThmBlvar}
Let $N\geq 2$ and $\mathcal{N}=\{1,...,N\}$. 
The horofunction boundary of the pseudo-normed 
space $\ell_{\var}(\mathcal{N},\R)$ is given by 
\begin{equation}\label{lvarhoroset}
	\partial_H\ell_{\var}(\mathcal{N},\R)= \left\lbrace 
						x\mapsto h_{\mu,\nu}^{\mathcal{I},\mathcal{J}}(x)
							\mathrel{\bigg\vert} 
							\begin{aligned}							
								&\emptyset\subsetneq\mathcal{I},\mathcal{J}\subsetneq\mathcal{N},\;
								\mathcal{I}\cap\mathcal{J}=\emptyset,\\
							 	&\mu\in\partial\R_{+}^{\mathcal{I}},\;\nu\in \partial\R_{+}^{\mathcal{J}}
							 \end{aligned} 	
						\right\rbrace,
\end{equation}
where $h_{\mu,\nu}^{\mathcal{I},\mathcal{J}}(x)=\topf(x_\mathcal{I}-\mu) - \botf(x_\mathcal{J}+\nu)$
for all $x\in \ell_{\var}(\mathcal{N},\R)$.
\end{thm}
\begin{proof}
Suppose that $h\in\partial_H\ell_{\var}(\mathcal{N},\R)$. 
Then there exists a sequence $\{y^{n}\}_{n\in\N}$ 
in $\ell_{\var}(\mathcal{N},\R)$ with $\norm{y^{n}}_\var \to \infty$ 
such that $\tau_\var(y^{n})$
converges pointwise to $h$ as $n\to\infty$.
By \Cref{ThmAlvar}, there exist 
$\emptyset\subsetneq\mathcal{I},\mathcal{J}\subsetneq\mathcal{N}$
with $\mathcal{I}\cap\mathcal{J}=\emptyset$, and 
$\mu\in \partial\R_{+}^{\mathcal{I}},\nu\in \partial\R_{+}^{\mathcal{J}}$ 
such that there is a subsequence $\tau_\var(y^{n_k})$ 
that converges pointwise to $h_{\mu,\nu}^{\mathcal{I},\mathcal{J}}$ as 
$k \to \infty $. Therefore $h=h_{\mu,\nu}^{\mathcal{I},\mathcal{J}}$ 
and so $\partial_H\ell_{\var}(\mathcal{N},\R)$
is contained in the set on the right-hand side of (\ref{lvarhoroset}).

Now, we need to show that for given 
$\emptyset\subsetneq\mathcal{I},\mathcal{J}\subsetneq\mathcal{N}$
with $\mathcal{I}\cap\mathcal{J}=\emptyset$, and 
$\mu\in \partial\R_{+}^{\mathcal{I}},\nu\in \partial\R_{+}^{\mathcal{J}}$, 
the function $h_{\mu,\nu}^{\mathcal{I},\mathcal{J}}$ 
is a horofunction on $\ell_{\var}(\mathcal{N},\R)$.
First we show that $h_{\mu,\nu}^{\mathcal{I},\mathcal{J}}$ 
belongs to $\overline{\ell_{\var}(\mathcal{N},\R)}^{H}$.
Let $(y^{n})_n$ be the sequence in $\ell_{\var}(\mathcal{N},\R)$ 
given by
\[
	y_{i}^{n}=\begin{cases} -n+\mu_i, & i\in\mathcal{I}, \\
							n-\nu_i, & i\in\mathcal{J},\\
							0,  & \mbox{ otherwise}.
				\end{cases}		
\]
Let $x\in\ell_{\var}(\mathcal{N},\R)$. 
Then, by (\ref{emblvar}) we have
\[
	\tau_\var(y^{n})(x)=\norm{x-y^{n}}_\var-\norm{y^{n}}_\var 
	\xrightarrow[n \to \infty]{} 
	\topf(x_\mathcal{I}-\mu) - \botf(x_\mathcal{J}+\nu)
	= h_{\mu,\nu}^{\mathcal{I},\mathcal{J}}(x).
\]
Hence $h_{\mu,\nu}^{\mathcal{I},\mathcal{J}}$ 
is an element of $\overline{\ell_{\var}(\mathcal{N},\R)}^{H}$.
It remains to show that $h_{\mu,\nu}^{\mathcal{I},\mathcal{J}}$ 
is not in $\tau_\var(\ell_{\var}(\mathcal{N},\R))$.
Suppose the contrary, so there exists $z\in\ell_{\var}(\mathcal{N},\R)$ 
such that $h_{\mu,\nu}^{\mathcal{I},\mathcal{J}}=\tau_\var(z)$.
For each $k$, define
\[
	x_{i}^{k}=\begin{cases} \mu_i, & i\in\mathcal{I}, \\
							-\nu_i, & i\in\mathcal{J},\\
							-k,  & \mbox{ otherwise}.
				\end{cases}			
\]
Then $h_{\mu,\nu}^{\mathcal{I},\mathcal{J}}(x^{k})=0$
for all $k$. However,
\[
	\tau_\var(z)(x^{k})=\norm{x^{k}-z}_\var-\norm{z}_\var 
	\centernot{\xrightarrow[k \to \infty]{}0},
\]
which is a contradiction. 
Therefore $h_{\mu,\nu}^{\mathcal{I},\mathcal{J}}$ 
belongs to $\partial_H\ell_{\var}(\mathcal{N},\R)$.
\end{proof}

\subsection{Hilbert's projective metric on $\R_{>0}^{\mathcal{N}}$}
We define \textit{Hilbert's projective metric} 
on $\R_{>0}^{\mathcal{N}}$ by
\[
	\dH(x,y):=\log\frac{\topf(x\odot y^{-1})}{\botf(x\odot y^{-1})}
\]
for all $x,y$ in $\R_{>0}^{\mathcal{N}}$ 
(see \cite{Birkhoff1957,Bushell1973_2,Bushell1973_1}). 
In fact, $\dH(\cdot,\cdot)$ is a pseudo-metric 
on $\R_{>0}^{\mathcal{N}}$. 
More precisely, $\dH(x,y)=0$ if and only if $x=\beta y$
for some $\beta > 0$. Also, 
$\dH(\alpha x,\beta y)=\dH(x,y)$ for all $\alpha,\beta>0$ 
and all $x,y\in\R_{>0}^{\mathcal{N}}$.
By (\ref{t-Log-Log}) and (\ref{varnorm}), it follows that 
\begin{align}
	\dH(x,y)&=\log \topf(x\odot y^{-1}) - \log \botf(x\odot y^{-1}) \nonumber\\
			&=\topf(\Log x - \Log y) - \botf(\Log x - \Log y) \nonumber\\
			&=\norm{\Log x - \Log y}_\var. \label{LOGiso}
\end{align}
In other words, the mapping $\Log$ is an isometry 
of $(\R_{>0}^{\mathcal{N}},\dH)$ into $\ell_{\var}^{\mathcal{N}}$. 
See \cite{Nussbaum1988,Lemmens_Nussbaum2012} for more details. 
The horofunction boundary of Hilbert's projective 
metric space $(\R_{>0}^{\mathcal{N}},\dH)$
is completely described by combining (\ref{LOGiso}) and \Cref{ThmBlvar} 
as follows.
\begin{cor}\label{horoHilbert}
The horoboundary for Hilbert's projective metric $\dH$ on $\R_{>0}^{\mathcal{N}}$ 
is given by 
\[
	\partial_H (\R_{>0}^{\mathcal{N}},\dH)= \left\lbrace 
							x\mapsto h_{u,v}(x)
							\mathrel{\bigg\vert} 
							\begin{aligned}
								&u,v\in \R_{+}^{\mathcal{N}},\;
								\topf(u)=1, \;\topf(v)=1,\\
								&u\odot v=0\cht
							\end{aligned}							
						\right\rbrace.
\]
where $h_{u,v}(x):=\log\topf(u \odot x)+\log\topf(v\odot x^{-1})$
for all $x\in\R_{>0}^{\mathcal{N}}$.
\end{cor}
\begin{proof}
Let $x\in\R_{>0}^{\mathcal{N}}$. For every $y^{n}\in\R_{>0}^{\mathcal{N}}$ 
we have
\begin{align*}
	\tau_{\dH}(y^{n})(x)&=\dH(x,y^{n})-\dH(\cht,y^{n}) \\
					&=\norm{\Log(x)-\Log(y^{n})}_\var - \norm{\Log(\cht)-\Log(y^{n})}_\var\\
					&=\tau_\var(\Log (y^{n}))(\Log(x)).
\end{align*}
Note that $\dH(\cht,y^{n})\to \infty$ as $n\to \infty$ 
if and only if $y^{n}\to\xi\in\partial\R_+^{\mathcal{N}}$ 
as $n\to \infty$.
The latter can be expressed equivalently by 
$\norm{\Log y^{n}}_\var\to \infty$ as $n\to \infty$.
By \Cref{ThmBlvar}, it follows that 
$h_\xi\in\partial_H (\R_{>0}^{\mathcal{N}},\dH)$ is
given by 
\[
	h_\xi(x)=h_{\mu,\nu}^{\mathcal{I},\mathcal{J}} (\Log(x))
	=\topf(\Log(x_\mathcal{I})-\mu)+\topf(-\Log(x_\mathcal{J})-\nu),
\]
where $\emptyset\subsetneq\mathcal{I},\mathcal{J}\subsetneq\mathcal{N}$ 
with $\mathcal{I}\cap\mathcal{J}=\emptyset$, and
$\mu\in\partial\R_{+}^{\mathcal{I}},\;\nu\in \partial\R_{+}^{\mathcal{J}}$.
Finally, consider $u=(u_i)_{i\in\mathcal{N}}$ and $v=(v_i)_{i\in\mathcal{N}}$ 
given by
\begin{equation*}
\begin{aligned}
	u_{i}=\begin{cases} \exp(-\mu_i), & i\in\mathcal{I}, \\
							0,  & i\in\mathcal{N}\setminus\mathcal{I},
				\end{cases}		
\end{aligned}
\begin{aligned}
\;,\;\;
\end{aligned}
\begin{aligned}
	v_{j}=\begin{cases} \exp(-\nu_j), & j\in\mathcal{J}, \\
							0,  & j\in\mathcal{N}\setminus\mathcal{J}.
				\end{cases}		
\end{aligned}
\end{equation*}
Then $u,v\in \R_{+}^{\mathcal{N}}$ 
with $\topf(u)=1$, $\topf(v)=1$ and $u\odot v=0\cht$. 
Hence, by (\ref{t-Log-Log}) it follows that 
\[
	h_\xi(x)=\log\topf(u \odot x)+\log\topf(v\odot x^{-1})=h_{u,v}(x).
\]
\end{proof}

\subsection{Perron's Theorem}
Let $N\geq 2$ and $\mathcal{N}=\{1,...,N\}$. 
Let $T=(T_{ij})\in\R^{\mathcal{N}\times \mathcal{N}}$ 
be a positive matrix, that is $T_{ij}>0$ for all $i,j\in\mathcal{N}$. 
Perron's theorem states that $T$ fixes a unique point 
in $\R_{>0}^{\mathcal{N}}/\R_{>0}$. We give here a new proof 
by applying the horofunction boundary of Hilbert's projective 
metric on $\R_{>0}^{\mathcal{N}}$.

Let $x\in\R_{>0}^{\mathcal{N}}$. 
For each $i,j\in\mathcal{N}$ we have
\begin{align*}
	(Tx)_i &:=\sum_{k\in\mathcal{N}}T_{ik}x_k \leq \topf((T_{ik})_{k\in\mathcal{N}})\sum_{k\in\mathcal{N}}x_k, \\
	(Tx)_j &:=\sum_{k\in\mathcal{N}}T_{jk}x_k \geq \botf((T_{jk})_{k\in\mathcal{N}})\sum_{k\in\mathcal{N}}x_k. 
\end{align*}
By \Cref{horoHilbert}, each element of the 
horofunction boundary $\partial_H (\R_{>0}^{\mathcal{N}},\dH)$ 
is of the form $h_{u,v}(x)=\log\topf(u \odot x)+\log\topf(v\odot x^{-1})$,
where $u,v\in \R_{+}^{\mathcal{N}}$ 
with $\topf(u)=1$, $\topf(v)=1$ and $u\odot v=0\cht$. 
Thus, 
\begin{equation}\label{horo}
	h_{u,v}(Tx) \leq \log 
	\max_{i,j}\bigg\{ u_i v_j \dfrac{\topf((T_{ik})_{k\in\mathcal{N}})}{\botf(T_{jk})_{k\in\mathcal{N}})}\bigg\}. 
\end{equation}
Let $r_{u,v}$ denote the term on the right-hand side of (\ref{horo}). 
Therefore, for every $x\in\R_{>0}^{\mathcal{N}}$, 
the sequence $\{Tx,T^{2}x,T^{3}x,...\}$ stays within 
the horoball $\mathcal{H}(h_{u,v},r_{u,v})$. 
It is well-known \cite{Karlsson_Metz_Noskov2006,Walsh2008} 
that horoballs for Hilbert's projective metric $\dH$ are convex subsets. 
It is also well-known \cite{Nussbaum1988,Lemmens_Nussbaum2012} 
that the norm topology and $\dH$-topology are equivalent 
in $\R_{>0}^{\mathcal{N}}/\R_{>0}$. 
By combining these facts and (\ref{horo}), 
we readily obtain the following.
\begin{lem}\label{compactcone}
The set
\[
	 C=\bigcap_{\substack{u,v\in\R_{+}^{\mathcal{N}} \\ 
	 \topf(u)=\topf(v)=1 \\ u\odot v=0\cht}} \mathcal{H}(h_{u,v},r_{u,v})
\]
is a nonempty convex subset of $\R_{>0}^{\mathcal{N}}$. 
Furthermore, $C$ is compact in $\R_{>0}^{\mathcal{N}}/\R_{>0}$ and 
is invariant under the positive matrix $T$, that is, $TC\subset C$.
\end{lem}
We can now consider $T$ as a self-mapping of 
the compact metric space $(C,\dH)$. 
In order to prove that $T$ fixes a unique point 
in $C\subset\R_{>0}^{\mathcal{N}}/\R_{>0}$ we will need the
following. 
\begin{lem}\label{contractive}
Let $x\neq y$ in $\R_{>0}^{\mathcal{N}}/\R_{>0}$. 
Then $\dH(Tx,Ty) < \dH(x,y)$.
\end{lem}
\begin{proof}
If $x\neq y$ in $\R_{>0}^{\mathcal{N}}/\R_{>0}$, 
then there exist $r,s\in\mathcal{N}$ with $r\neq s$ such
that 
\[
	\botf(x\odot y^{-1})
	=\frac{x_r}{y_r}<\frac{x_s}{y_s}
	=\topf(x\odot y^{-1}).
\]
On the other hand, for every $i$,
\begin{align*}
		(Tx)_i &= T_{ir}x_r + T_{is}x_s + \sum_{\mathclap{\substack{k\in\mathcal{N} \\ k\neq r,s}}}T_{ik}x_k\\
				& < T_{ir}y_r\frac{x_s}{y_s} + T_{is}\frac{x_s}{y_s}y_s + 
				\sum_{\mathclap{\substack{k\in\mathcal{N} \\ k\neq r,s}}}T_{ik}\frac{x_k}{y_k}y_k \\
				& \leq T_{ir}y_r\frac{x_s}{y_s} + T_{is}\frac{x_s}{y_s}y_s + 
				\sum_{\mathclap{\substack{k\in\mathcal{N} \\ k\neq r,s}}}T_{ik}\frac{x_s}{y_s}y_k \\
				& =\frac{x_s}{y_s}(Ty)_i.
\end{align*}
The above implies that $\topf(Tx\odot (Ty)^{-1}) < \topf(x\odot y^{-1})$. 
In a similar way we can show that
$\botf(Tx\odot (Ty)^{-1})>\botf(x\odot y^{-1})$. 
Therefore,
\[
	\dfrac{\topf(Tx\odot (Ty)^{-1})}{\botf(Tx\odot (Ty)^{-1})} 
	< \dfrac{\topf(x\odot y^{-1})}{\botf(x\odot y^{-1})},
\]
and the result follows.
\end{proof}
\begin{rem}
Samelson \cite{Samelson1957} gives a different proof of \Cref{contractive} 
by applying projective properties of cross-ratios, which appear 
in Hilbert's original definition of $\dH$.
\end{rem}
Finally, by combining \Cref{compactcone} and \Cref{contractive} 
and applying Edelstein's fixed-point theorem \cite{Edelstein1962} 
we obtain the following.
\begin{cor}[Perron's theorem]
There exists a unique point $x^{*}$ in 
$C\subset\R_{>0}^{\mathcal{N}}/\R_{>0}$ 
such that $T(x^{*})=x^{*}$.
\end{cor}

\subsection*{Acknowledgements}
I would like to thank Prof. Anders Karlsson for suggesting the research topic, as well
as for his constant guidance, support and encouragement. I would also like to thank
Prof. Kalle Kyt\"ol\"a and Prof. Olavi Nevanlinna for several helpful comments and
suggestions.

\end{document}